\RequirePackage{fix-cm}
\documentclass[smallextended]{svjour3}       
\smartqed  

\usepackage{graphicx}
\usepackage{epstopdf}
\usepackage[font=scriptsize,labelfont=bf]{caption}
\usepackage{dsfont}
\usepackage{float}
\usepackage[absolute]{textpos}
\usepackage{setspace}
\usepackage{mathptmx}      
\usepackage{amsmath,amssymb,amsfonts}
\usepackage[authoryear,round]{natbib}

\spnewtheorem{assumption}{Assumption}{\bf}{\it}

\newcommand{\R}{\mathbb{R}}
\newcommand{\N}{\mathbb{N}}

\begin{document}

\title{Pareto-Optima for a Generalized Ramsey Model
}


\author{L. Frerick \and G. M{\"u}ller-F{\"u}rstenberger \and \\ E.W. Sachs \and L. Somorowsky}


\institute{L. Frerick \and E.W. Sachs \and L. Somorowsky \at Department of Mathematics, Trier University, 54286 Trier, Germany
\and 
G. M\"uller-F\"urstenberger \at Department of Economics, Trier University, 54286 Trier, Germany
}

\date{Received: date / Accepted: date}

\maketitle

\begin{abstract}
We consider a Ramsey model with several households with heterogeneous preferences who are able to borrow capital to each other. Since the capital constraints of one household then depends on the others' capital, one can no longer optimize each household's welfare individually. This problem formulation leads to a Pareto optimization problem. We consider existence and first order optimality conditions as well as some numerical results.

\end{abstract}

\section{Introduction}
Our paper deals with a variation of the neoclassical growth model developed by \citet{ramsey} in the twenties of the last century.  The interesting part in the Ramsey model is the endogenous saving rate that is determined by the welfare optimization of households under a given market interest rate. Ramsey himself did not restrict the capital stock held by one of the individuals in the respective economy, but this capital stock is often constrained in the literature. Actually, Ramsey's model has been modified and analyzed in various ways. A vast literature exists that deals with different types of equilibria, with finite or infinite time horizons in a continuous or a discrete setting, considers taxation \citep{sorger} or a growing population \citep{acemoglu}, models a central planning authority \citep{cass} or studies Pareto optimization for a set of different households that are connected by their initial capital distribution \citep{van}.

In this paper we consider a group of $H$ economic agents (or households) with heterogeneous preferences. Each agent is endowed with an initial capital asset ($a_0^h $), and receives an income stream from supplying its labor-force ($l_t^h $).  Income from labor and capital is split between consumption and savings, where ``savings'' can be negative.  Agents seek to maximize intertemporal welfare, the time horizon is assumed finite ($T\in \N$). The number of agents is assumed small relative to the total population, hence an agent's activity has no impact on commodity and asset prices. 

Although this set-up is close to the standard Ramsey problem, where agents smooth their consumption path by shifting income over time via the capital market, it differs significantly. We now assume that the agents have formed a coalition to establish a capital sub-market. There they can borrow or lend to each other at world market interest rates, but the coalition's net debt-position must be non-negative in every point in time. Therefore some agents may have negative capital assets, but the coalition is free of debt when considered from the outside. In other words, the coalition makes sure that the default of individual debtors has no repercussions on the world markets. Since each agent's capital depends on the capital of the other agents and they altogether must have a nonnegative capital stock, we have to optimize their welfare simultaneously, which leads to a vector or Pareto optimization problem.

The real world situation we have in mind is a small country. Because of financial disruption, foreign investors lost confidence in its solvency.  The country is still embedded in the global capital market system, but its net-debt position must now be positive at any point in time. In this way, the outstanding loans may serve as collateral for outstanding debt, hence the country cannot default as a whole on its foreign debtors.\footnote{ We do not discuss in this paper, how this is framed into institutions.}

 Another example is a credit union, where its members team up to lend and borrow to each other at world market conditions but do not need to fulfill all regulations of global capital markets.  Also {\em peer-to-peer banking} serves as another example for our model. Via online platforms private individuals borrow money to other ones. Our idea is, that individual agents are allowed to incur debt (what weakens the market imperfection) but the community as a whole has to hold a non-negative capital stock. Like \citet{beckerlib}, we avoid the market imperfection that forbid the households to incur debt and define a new capital restriction instead.

In order to ensure the existence of a solution of the households optimization problem (and its uniqueness), the capital stock of the households is often constrained. That is why borrowing constraints are introduced which are huge market imperfections in the context of economics. The most common way to constrain the household's capital stock is to bound it from below by zero \citep[see][]{alt,becker,beckera,beckerb,sorger}. 

An approach to weaken the assumption of a non-negative capital stock of each household in an infinite time horizon model is the {\em liberal borrowing constraint} of \citet{beckerlib}. In this model households are allowed even to hold negative capital stocks only if they can cover their loans by future wage incomes for a given finite time horizon. 

There have been various approaches to Pareto optimization in the literature. \citet{van} and \citet{lucas} both showed the existence of a Pareto optimal consumption path using scalarization and Debreu's theory of a valuation equilibrium \citep[see][]{debreu}. In our context, we analyze the households optimization problem in the context of vector optimization theory. We will use the definition of a minimal element of a set according to \citet{jahn} which is equivalent to the commonly known definition of a Pareto optimum in the finite dimensional real space. The vector optimization approach allows us to optimize independently from a priori defined scalarization weights, which is a generalization of the household's optimization problem in the Ramsey model. If one is interested in a particular Pareto solution from the (possibly large) set of Pareto solutions, there is need of a ``(human) decision maker who knows the problem domain'' or a technique to choose the ``best'' Pareto solution in an appropriate way \citep{klamroth}.

The main issue in this paper is a mathematical rigorous theory about the existence and characterization of Pareto optimal allocations for the households' sector in such a setting. We assume that agents can have negative capital assets ($a^h_t \in\R$) for all periods, however the aggregated capital stock has to be nonnegative at all times. This is called the generalized Ramsey model with partially accepted default. Moreover, we consider the case where no agent is allowed to exit the economy at $T+1$ with negative savings. We call this the generalized Ramsey model without default. In the paper we use mathematical optimization theory to show the existence of optimal points in a vector optimization context. Then we formulate necessary and sufficient optimality conditions and interpret them in the two special contexts we alluded to earlier. To our knowledge this is new to the literature.

We model the household's sector analogously to \citet{becker}, i.e. we consider an economy with many households that differ in preferences according to time and utility.  \citet{becker} has shown that -under certain assumptions- in the long-run the most patient household holds all the money - a conjecture of Ramsey that was proven here for the first time. However, there are variants of the Ramsey model where the steady state is non-degenerate \citep[see][]{sarte,lucas}. In order to keep the model as general as possible, we allow different discount factors and instantaneous utilities. The analysis of a possible impact on the steady state in our generalized model with an infinite time horizon is part of our future work.

In the second section we define the underlying optimization problem and state various assumptions. An existence theorem is proven in section 3. The necessary and sufficient optimality conditions are derived in section 4 and they are discussed in various lemmas and corollaries. Section 5 deals with a slightly modified problem and section 6 concludes with some numerical results.


\section{Problem Setting as a Vector Optimization Problem}

In this paper, we formulate a version of the well-known Ramsey Equilibrium Model that allows for households with different preferences according to time and utility.  We consider the Ramsey model with a discrete set of households, a single capital good and a competitive labor market. Time is measured in discrete periods $t\in\{0,...,T\}$, $T\in\N$. We assume the households to be heterogeneous with respect to utility and time preferences. 

We denote the instantaneous utility function of each household  by 
\[u^h:(0,\infty)\to\R, \quad h\in\{1,...,H\},\ H\in\N,\] 
and the respective time discount factor by $\beta^h\in(0,1)$.
Households are assumed to discount future utility exponentially \citep[p. 180]{acemoglu}. Every household is endowed with $l^h\in\R_+$ units of labor per period which it supplies in an inelastic way. We denote the wage rate in period $t$ by $\omega_t\ge 1$ and the interest rate by $r_t>0$ which we assume both to be given. 

These assumptions have some consequences: First of all, there do exist differences in income between the households due to differences in labor capacity. Second, the evaluation of lifetime utility depends on the the households consumption profile and the amount of labor it is able to supply. And thirdly, the only variable in the production sector is the aggregated capital the households supply since labor is treated as a fixed factor.

We assume the initial capital endowment of every household to be given by $a_0^h\in\R$ such that the aggregated initial capital endowment satisfies
\[\sum_{h=1}^H a_0^h>0.\]
Furthermore, we assume that the capital stock $a^h_t$ of every household $h\in\{1,...,H\}$ in period $t\in\{0,...,T\}$ declines by a factor $\delta\in(0,1)$ and the investment  increases by a factor $\tau\ge 1$ during every period. The consumption of household $h$ in period $t$ is denoted by $c_t^h > 0$. Hence, the capital stock held by household $h$ at the beginning of period $t+1$ is given by the following equation
\[a_{t+1}^h=\tau \omega_t l^h+(\tau(1+r_t)-\delta)a_t^h-\tau c_t^h.\] 
In contrast to many well-established models like in \citet{beckera}, we do allow negative capital stocks in all periods. This means that single households can incur debt, but we assume that the aggregated capital stock of all households remains non-negative over all time periods, i.e.
\begin{equation} \label{positive}
\sum_{h=1}^H a_t^h
\ge 0, \quad t=1,...,T+1.
\end{equation}

Moreover, we postulate that the households are competitive agents that perfectly foresee the sequence of factor returns, $\{r_t,\omega_t\}_{t=0}^T$. Then, every household has to solve the following utility maximization problem:
\begin{equation}\label{Ph}
\begin{split}
\max_{(c^h,a^h)} \sum_{t=0}^T (\beta^h&)^t u^h(c_t^h), \quad h=1,...,H,\\
s.t.\quad  a_{t+1}^h&=\tau \omega_t l^h+(\tau(1+r_t)-\delta)a_t^h-\tau c_t^h, ~~ h=1,...,H, ~~ t=0,...,T,\\
\sum_{h=1}^H a_t^h&\ge 0, \quad t=1,...,T+1, \qquad c^h_t > 0,\quad t=0,...,T,\\
 \text{with} ~~ \sum_{h=1}^H a_0^h&>0,\ (a_0^h)_{h=1}^H\text{ given}.
\end{split}
\end{equation}
The constants satisfy the following conditions:
\begin{equation} \label{constants}
  r_t, l^h > 0,~~ \omega_t, \tau \geq 1,~~ \beta^h, \delta \in (0,1). 
  \end{equation}
If we replace the condition (\ref{positive}), that only the aggregate capital stock is non-negative, by the usual condition in the literature \citep[see for example][]{beckera}  that this non-negativity holds for the capital stock of each household, i.e. 
\[ a^h_t \geq 0, \quad h=1,...,.H, ~~ t=1,...,T+1 \]
 at any time instance, then the $H$ optimization problems in (\ref{Ph}) are decoupled and can be solved separately from each other. 
 
The fact that a household can accumulate debt, if they altogether have non-negative accumulated capital stock, makes this optimization problem more interesting but also mathematically more challenging. The optimization problem (\ref{Ph}) is a multi-objective optimization problem and we use the notation of Pareto-optimality to define solutions of this optimization problem.
 
For notational purposes, we rewrite the $H$ optimization problems in (\ref{Ph}) as a vector optimization problem in order to find a Pareto optimal solution for all households. Let us define the vector valued objective functional
\begin{equation}
\begin{split}
\varphi&:\R^ {2H(T+1)}\to \R^H,\\
\varphi_h(z)&:=-\sum_{t=0}^T(\beta^h)^tu^h(c_t^h),\quad h=1,...,H
\end{split}
\end{equation}
with variables 
\[z=(c_0^1,c_1^1,...,c_T^1,...,c_0^H,...,c_T^H,a_1^1,...,a^1_{T+1},...,a_1^H,...,a^H_{T+1})^T \in \R^ {2H(T+1)}\] 
and a set of functions
\[u^h:(0,\infty)\to\R ,\quad h=1,...,H\] 
which satisfy the so-called Inada Condition:
\begin{assumption} \label{ass_inada} [Inada Condition]
For $h=1,...,H$ let $u^h:(0,\infty)\to\R$ be twice continuously differentiable and the following conditions hold:
\begin{equation} \label{Inada}
\begin{array}{c}
(u^h)'(c)>0,\ (u^h)''(c)<0, ~~c>0, \\[2mm]
\lim_{c\to\infty}(u^h)'(c)=0,\ \lim_{c\to 0}(u^h)'(c)=+\infty, \quad h=1,...,H.
\end{array}
\end{equation}
\end{assumption}
This condition guarantees that households choose a positive consumption at each point of time \citep[see for example][]{beckerb} which will be essential in context of the first order conditions later on. On the other hand, this also implies that the feasible set is not necessarily compact which has to be considered in an existence proof.

In order to formulate the constraints, we define the following mappings
\[f:\R^ {2H(T+1)}\to \R^{H(T+1)},~~~g:\R^{2H(T+1)}\to \R^{T+1}\]
by
\[f_t^h(z):=a_{t+1}^h-\tau \omega_t l^h-(\tau(1+r_t)-\delta)a_t^h+\tau c_t^h, \quad t=0,...,T,~~h=1,...,H\]
and
\[ g_t(z):=-\sum_{h=1}^Ha^h_t, \quad t=1,..,T+1.\]
The components of the function $f$ are set as follows
\[f=(f_1,...,f_{H(T+1)})^T=(f_0^1,...,f_T^1,...,f_0^H,...,f^H_T)^T.\]

Then, the aggregated optimization problem of all households can be written as a vector optimization problem in the following way.

\begin{definition}
Let the vector $(a_0^1,...,a_0^H)$ be given such that $\sum_{h=1}^Ha_0^h>0$. Then we define the vector optimization problem
\begin{equation}\label{P}
\min_{z\in\mathcal{U}}\varphi(z)
\end{equation}
where $\mathcal{U}$ denotes the set of feasible points
\begin{equation*}
\begin{split}
\mathcal{U}:=\{z\in\hat{\mathcal{U}}:\ f_t^h(z)&=0\quad (t=0,...,T;\ h=1,...,H),\\ 
g_t(z)&\le 0 \quad (t=1,...,T+1)\}
\end{split}
\end{equation*} 
with
\[\hat{\mathcal{U}}:=\{z\in\R^{2H(T+1)}:\ z_1,...,z_{H(T+1)} > 0\}.\]
\end{definition}

It is easy to prove that the feasible set is nonempty which is stated in the following lemma.
 \begin{lemma} \label{nonempty}
The feasible set $\mathcal{U}$ is nonempty.
\end{lemma}
This can be seen easily from equation (\ref{Ph}) by choosing
\[ c_t^h = \omega_t l^h > 0, \quad a^h_{t+1} = (\tau(1+r_t) - \delta) a^h_t.\]
\\
In the sequel of the paper we consider Pareto optimal points which are defined as follows:
\begin{definition}\label{po}
We consider a solution $\overline{z}\in\mathcal{U}$ to be Pareto optimal if and only if there does not exist any vector $z\in\mathcal{U}$ such that
\[\varphi_h(z)\le \varphi_h(\overline{z})\text{ and }\varphi_{h^*}(z)<\varphi_{h^*}(\overline{z})\text{ for at least one }h^*\in\{1,...,H\}.\]
\end{definition}
In the following sections we will consider questions like the existence of Pareto optimal points and  optimality conditions. The latter ones will be interpreted in context of the application for the Ramsey model.

\section{Existence Theorem}

The proof of existence of Pareto optima turns out to be somewhat delicate, since in general the set of of feasible points $\mathcal{U}$ is not a closed set. This cannot be fixed easily, since the objective functions $u^h(c)$ are defined only for positive $c$ and is not defined for $c=0$ like the log-function, a typical utility function. 

This particular difficulty is often ignored in the literature and it is our goal to provide a rigorous proof for this problem. In order to prove the existence of a Pareto optimal solution of (\ref{P}) we use an existence theorem in \citet{jahn} and the technique of level sets.

 The definition of a Pareto optimum given above is akward to use from a mathematical point of view, we introduce the definition of a minimal element according to \citet{jahn}.
 
\begin{definition} \label{minpt}
Define the natural ordering cone of $\R^H$ as
\[C_{\R^H}:=\{y\in\R^H:\ y_i\ge 0\quad  i=1,...,H \}.\]
An element $y^*$ of $\varphi(\mathcal{U})$ is called a minimal element of $\varphi$ over $\mathcal{U}$ if it satisfies
\[(\{y^*\}-C_{\R^H})\cap \varphi(\mathcal{U})=\{y^*\}.\]
\end{definition}
This definition is in analogy to the definition \ref{po}, however formulated in mathematical terms.\\

\begin{theorem}\label{existence}
There exists a Pareto optimal solution of  (\ref{P}).
\end{theorem}

\begin{proof}
According to \citet[p. 139]{jahn} we define a section $S_{y}$ of the image $S:=\varphi(\mathcal{U})$ as
\[
S_{y}:=(\{y\}-C_{\R^H})\cap \varphi(\mathcal{U})\neq\emptyset
\]
for an arbitrary $y\in\R^H$. Theorem 6.3 in \citet[p. 140]{jahn} guarantees the existence of a minimal element of $\varphi(\mathcal{U})$, if the set $\varphi(\mathcal{U})$ has a compact section for some $y\in\R^H$. 

By Lemma \ref{nonempty} the set of feasible points $\mathcal{U}$ is not empty. Hence there exists an element $\overline{z}\in \mathcal{U}$ and $\overline{y}:=\varphi(\overline{z})\in\varphi(\mathcal{U})$ and we can define the section
\begin{equation} \label{sec}
S_{\overline{y}}:=\varphi(\mathcal{U})\cap (\{\overline{y}\}-C_{\R^H}).
\end{equation}
We rewrite for given initial values $a_0^h$ the feasible set as
\begin{equation*}
\begin{split}
\mathcal{U}:=\{(&c_0^1,c_1^1,...,c_T^1,...,c_0^H,...,c_T^H,a_1^1,...,a^1_{T+1},...,a_1^H,...,a^H_{T+1})^T \in \R^ {2H(T+1)}:\\ &a_{t+1}^h=\tau \omega_t l^h+(\tau(1+r_t)-\delta)a_t^h-\tau c_t^h\quad t=0,...,T,\ h=1,...,H,\\ 
&\sum_{h=1}^H a_t^h\ge 0, \ t=1,...,T+1, \quad c_t^h>0,  \ t=0,...,T \}
\end{split}
\end{equation*}
From
\[ a_{t+1}^h=\tau \omega_t l^h +(\tau(1+r_t)-\delta)a_t^h -\tau c_t^h = \xi_t^h+\gamma_t a_t^h-\tau c_t^h \]
with $\xi_t^h:=\tau\omega_t l^h$ and $\gamma_t:=(\tau(1+r_t)-\delta)$
we obtain by induction
\begin{equation} \label{recurs}
a_{t+1}^h=\sum_{s=0}^t \xi_s^h \prod_{v=s+1}^t \gamma_v + \prod_{s=0}^t \gamma_s a_0^h - \tau(\sum_{s=0}^t c_s^h \prod_{v=s+1}^t \gamma_v)
\end{equation}
with $\prod_{\emptyset}\gamma_v:=1$. The constants in assumption \ref{constants} yield
$\gamma_t\ge 0$, 
$\xi_t^h\ge 0$ and $c_t^h > 0$ for all $h\in\{1,...,H\}$ and $t\in\{1,...,T\}$ and hence for some constant $a_{max}$
\begin{equation} \label{amax}
 a_{t+1}^h \le \sum_{s=0}^t \xi_s^h \prod_{v=s+1}^t \gamma_v + \prod_{s=0}^t \gamma_s a_0^h \leq a_{max},
\quad t=0,...,T, ~~ h=1,...,H. 
\end{equation}
The bound from below for $a_t^h$ follows from (\ref{positive}) via
\[ - (H-1) a_{max} \le -\sum_{k=1\atop{k\neq h}}^H a^k_t \le a_t^h \quad h=1,...,H,~~ t=1,...,T+1.\]
The boundedness of the $c_t^h$ can be concluded from (\ref{positive}), (\ref{recurs}) and (\ref{amax}):
\[
 0\le \sum_{h=1}^H a_{t+1}^h =\sum_{h=1}^H\left( a_{max}  - \tau(\sum_{s=0}^t c_s^h \prod_{v=s+1}^t \gamma_v) \right)\]
 This holds if and only if 
\[ \quad \sum_{h=1}^H \sum_{s=0}^t c_s^h \prod_{v=s+1}^t \gamma_v  \le \frac{H} {\tau} a_{max}.\]

Since $c_t^h$ are positive for all $t$ and $h$ we obtain that all $c_t^h$ are also bounded from above. 

We have shown that $\mathcal{U}$ is bounded and therefore also $\phi(\mathcal{U})$ and consequently is the section $S_{\bar y}$ also a bounded set. Unfortunately, the closedness cannot be concluded in this way, because, due to $c_t^h>0$, the set $\mathcal{U}$ is not closed in general. However, we can prove directly that the section $S_{\overline{y}}$ is closed using a level set argument because
\[ S_{\overline{y}}:=\varphi(\mathcal{U})\cap (\{\overline{y}\}-C_{\R^H})
=\{\varphi(z): z\in\mathcal{U}, \varphi_h(z) \leq \overline{y}_h\,\ h=1,...,H\}. \]
Since we have previously shown that the $c_t^h$ are bounded from above, i.e. for some $c_{max} > 0$ we have
\[ c_t^h \leq c_{max}, \quad h=1,...,H,~~ t=1,...,T,\]
the inequality $\varphi_h(z) \leq \overline{y}_h$ is only effective for those components of $z$ which are $c_t^h$, because
\[ -\sum_{t=0\atop{t\neq t^*}}^T (\beta^h)^t u^h(c_{max}) - (\beta^h)^{t^*} u^h(c_{t^*}^h) \leq -\sum_{t=0}^T (\beta^h)^t u^h(c_t^h) = \varphi_h(z) \le \overline{y}\]
which yields with the monotonicity of $u^h$
\[ c_{t^*}^h\ge (u^h)^{-1}\left( \frac{1}{(\beta^h)^{t^*}}\left(-\sum_{t=0\atop{t\neq {t^*}}}^T (\beta^h)^t u^h(c_{max})-\overline{y} \right)\right):=\vartheta_{t^*}^h>0 . \]

We define $\vartheta\in\R^{H(T+1)}$ as $\vartheta:=(\vartheta_0^1,...,\vartheta_T^1,...,\vartheta_0^H,...,\vartheta_T^H)$
Then we can rewrite the section $S_{\overline{y}}$ as
\[ S_{\overline{y}}=\{\varphi(z):\ z\in\mathcal{U}, z_i\ge \vartheta_i \ i=1,...,H(T+1)\}\]
which is bounded and closed, hence compact.
\end{proof}

\section{Necessary and Sufficient Optimality Conditions}
In this section, we derive necessary and sufficient optimality conditions for Pareto optimal points. Also here we will follow \citet[pp. 152-167]{jahn} for a rigorous derivation. First, we verify the constraint qualification for a minimal solution $\overline{z}$ of the problem. 
\begin{lemma}
$\varphi, f$ and $g$ are continuously partially differentiable at $\overline{z} \in \mathcal{U}$ with 
\begin{itemize}
\item $\nabla \varphi_h(\overline{z})$ the vector in $\R^{2(T+1)H}$ with $-(\beta^h)^t(u^h)'(\overline{c}_t^h)$ at position  $(T+1)(h-1)$ to $((T+1)h)$ and else $0$
\item $\nabla g_t(\overline{z})=(0,...,0,-1,0,...,0,-1,0,...,0)^T$, with  $H$ entries different from zero at position $((H+j)(T+1)+t)$ for $j=0,...,T$.
\item $\nabla f_t^h(\overline{z})=(0,...,0,\tau,0,...,0,-\xi_t,1,...,0)^T\in \R^{2H(T+1)}$ where the entries different from zero correspond to $c_t^h,\ a^h_t$ and $a^h_{t+1}$ and \\$\xi_t:=(\tau(1+r_t)-\delta)$.
\end{itemize}
\end{lemma}
We denote in the usual way the active set of inequality constraints in $\overline{z}\in\mathcal{U}$ as
\[\mathrm{I}(\overline{z}):=\{j\in\{1,...,T+1\}:\ g_j(\overline{z})=0\}.\]
We will need the following statements if we want to apply a theorem of \cite{jahn} to derive the optimality conditions:
\begin{lemma}
The gradients  
\[\nabla g_j,\nabla f_1,...,\nabla f_{H(T+1)}\]  
with $\ j\in\mathrm{I}(\overline{z})$ are linearly independent. 
\end{lemma}
\begin{proof}
The only solution of
\[\sum_{i=1}^{H(T+1)} \lambda_i\nabla f_i(\overline{z})+\sum_{\ j\in\mathrm{I}(\overline{z})}\varsigma_j\nabla g_j(\overline{z})=0\]
is $\lambda_i=0\ \forall i\in\{1,...,H(T+1)\}$ and $\varsigma_j=0\ \forall \ j\in\mathrm{I}(\overline{z})$.
Note that there are no entries different from zero in the first $H(t+1)$ lines of $\nabla g_j(\overline{z})$ which yields that $\lambda_i=0$ for all $i=1,...,H(T+1)$. This in turn leads to $\varsigma_j=0$ for all $j\in\mathrm{I}(\overline{z})$.
\end{proof}

Moreover this yields the following lemma:

\begin{lemma}
For every minimal solution $\overline{z}$ in $\mathcal{U}$ there exists some $z\in\R^{2H(T+1)}$ such that 
\[\nabla g_j(\overline{z})^T(z-\overline{z})<0 \text{ for all } j\in\mathrm{I}(\overline{z}).\] 
\end{lemma}

\begin{proof}
Since $\nabla g_t(\overline{z})=(0,...,0,-1,0,...,0,-1,0,...,0)^T$ with  $H$ entries different from zero at position $((H+j)(T+1)+t)$ for $j=0,...,T$, every $z\in\R^{2H(T+1)}$ with\\
$(z_t-\overline{z}_t)>0$ for all $t=1,...,T+1$ satisfies the desired inequality.
\end{proof}

Hence the constraint qualification for this problem is fulfilled and we can derive the necessary and sufficient optimality conditions.

For $\theta \in \R^H_+$, $\nu\in \R^{H(T+1)}_+$ and $\varrho\in \R^{H(T+1)}$ we define the real-valued Lagrangian of (\ref{P}) as

\begin{equation}\label{Lagr}
\mathcal{L}(z,\theta,\nu,\varrho):=\sum_{i=1}^H \theta_i \varphi_i(z)+\sum_{j=1}^{T+1}\nu_jg_j(z)+\sum_{k=1}^{(T+1)H} \varrho_k f_k(z).
\end{equation}
We define
\[\varrho:=(\lambda_0^1,...,\lambda_{T}^1,...,\lambda^H_0,...,\lambda_T^H)\in \R^{(T+1)H}.\]
Then (\ref{Lagr}) can be rewritten as
\begin{equation}\label{Lag}
\begin{split}
\mathcal{L}(z;\theta,\nu,\varrho)=&-\sum_{h=1}^H  \sum_{t=0}^T \theta_h (\beta^h)^t u^h(c^h_t) - \sum_{h=1}^H\sum_{t=1}^{T+1} \nu_t a_t^h  \\
& + \sum_{h=1}^H\sum_{t=0}^T \lambda_t^h (a_{t+1}^h-\tau \omega_tl^h -(\tau(1+r_t)-\delta)a_t^h+\tau c_t^h).
\end{split}
\end{equation}


Hence the  first order conditions are given in the following form

\begin{theorem} \label{thm-noc}
A vector $\overline{z}$ is Pareto-optimal if and only if there exist multipliers $\theta = (\theta_1,...,\theta_H)\in \R^H$ (where at least one $\theta_h\neq 0$), $\lambda=(\lambda_0^1,...,\lambda_{T}^1,...,\lambda^H_0,...,\lambda_T^H)$ $\in\R^{H(T+1)}$ and $\nu=(\nu_1,...,\nu_{T+1})\in \R^{T+1}$ such that

\begin{equation*}
\begin{split}
\text{(I)}\quad\frac{\partial \mathcal{L}(\overline{z};\theta,\nu,\varrho)}{\partial {c}_t^h}&=-\theta_h(\beta^h)^t(u^h)'(\overline{c}_t^h)+\tau \lambda_t^h=0,\quad t=0,...,T,~~ h=1,...,H\\[3mm]
\text{(II)}\quad\frac{\partial \mathcal{L}(\overline{z};\theta,\nu,\varrho)}{\partial {a}_t^h}&=-\nu_t+\lambda_{t-1}^h-\lambda_t^h(\tau(1+r_t)-\delta) = 0,\quad t=1,...,T,~~h=1,...,H\\[3mm]
\text{(III)}\quad\frac{\partial \mathcal{L}(\overline{z};\theta,\nu,\varrho)}{\partial {a}_{T+1}^h}&=  \lambda_T^h-\nu_{T+1}= 0,\quad h=1,...,H\\[3mm]
\text{and the compl}&\text{ementary slackness condition holds}\\[3mm]
\text{(IV)}\hspace{0.5cm}-\sum_{h=1}^H\nu_t\overline{a}_t^h&=0,\quad t=1,...,T+1\\[3mm]
\text{(V)}\hspace{1.4cm}\theta_h,\nu_t&\ge 0, \quad t=1,...,T+1,~~h=1,...,H.
\end{split}
\end{equation*} 
\end{theorem}
The proof follows directly from \citet[pp. 165-167]{jahn}. Note that our problem is a convex optimization problem, since the objective function is convex and the constraints are linear.

We can analyze these optimality conditions a bit further and obtain the following results.
\begin{corollary} \label{cor1}
The following statements are true:\\
a) The multipliers $\lambda_t^h$ are independent of the households, i.e. $\lambda_t^h = \lambda_t$ for all $t=0,...,T$.\\
b) The multipliers $\theta_h >0 $ are positive for all households $h=1,...,H$.\\
c) We have at the final time $\sum_{h=1}^H \overline{a}_{T+1}^h=0$.
\end{corollary}
\begin{proof}
a) Equation (III) yields that $\lambda_T^h=\nu_{T+1}$ for all $h=1,...,H$, hence independent of $h$. Since $\nu_{t}, t=1,...T$ is also independent of $h$, using equation (II) recursively, we obtain that all multipliers $\lambda^h_t$ are independent of $h$ for $t=0,...,T$.

b) Assume there is a household $h^*$ such that $\theta_{h^*}=0$. Then equation (I) then yields $\lambda^{h^*}_t=0$ for all points of time $t=0,...,T$. Since the $\lambda^{}_t=0$ are independent of $h$ by a), this implies by equation (I) again that $\theta_h = 0 $ for all $h=1,...,H$.

c) If we assume $\nu_{T+1}=0$, (III) implies that $\lambda_T^h=0$ for all $h$. With equation (I) we then obtain for all $h$
\[\theta_h(\beta^h)^T(u^h)'(\overline{c}^h_T)=0.\]
Due to $(u^h)'>0$ by assumption, this yields $\theta_h=0$ in contradiction to $\theta_h>0$ for all $h$.
Hence $\nu_{T+1}\neq 0$ and by the complementary slackness condition
\[\sum_{h=1}^H \overline{a}_{T+1}^h=0.\]
\end{proof}
From an application's point of view, the Corollary \ref{cor1} can be interpreted as follows. 

To understand part a) better, note that in problem (\ref{Ph}) during the optimization phase the capital stock $a^h_t$ of each individual household is not relevant but rather the sum of all households
\[ \alpha_t = \sum_{h=1}^H a^h_t, ~~~ t=0,...,T, \]
since for each time instant $t$ the $H$ equality constraints can be replaced by a single constraint
\[ \alpha_{t+1} =\tau \omega_t \sum_{h=1}^H l^h+(\tau(1+r_t)-\delta)\alpha_t-\tau \sum_{h=1}^H  c_t^h ~~ t=0,...,T.
\]
Therefore, there is per time instance only one Lagrange multiplier $\lambda_t$ which is independent of $h$.

Part b) shows that every household is influencing the minimal solution, or in other words, if one household is omitted, the minimal solution could change immediately. The condition in c) tells us that at the final time, everything will be consumed, which is obvious from the application, but could also be concluded  from the  optimality conditions.

In the following statement we have collected all the previous information about the minimal solution:

\begin{theorem} \label{thm-prop}
Consider the original optimization problem (\ref{Ph}) and let $\overline{c}^h_t, \overline{a}^h_t$ be minimal solutions. 
For given initial values $\overline{a}_0^1,....,\overline{a}_0^H$ with $\sum_{h=1}^H \overline{a}_0^h>0$  there exist multipliers $\theta = (\theta_1,...,\theta_H)$, all positive,
and $\nu=(\nu_1,...,\nu_{T+1})$, all nonnegative, such that 
the optimal consumption $\overline{c}_t^h$ can be obtained by solving recursively backwards
\begin{equation} \label{backward}
\lambda_{t-1} = \lambda_t(\tau(1+r_t)-\delta) + \nu_t, ~~\quad t=T, T-1, ..., 1.
\end{equation}
with the final condition
\begin{equation} \label{final}
\lambda_T = \nu_{T+1}.
\end{equation}
The optimal consumption is given by
\begin{equation}
\overline{c}^h_{t} = ((u^h)')^{-1} \left( \frac{\tau~ \lambda_{t}}{(\beta^h)^t ~\theta_h} \right) ~~ h=1,...,H,~~ t=0,...,T
\end{equation}
and the optimal capital stock $\overline{a}_t^h$ can be computed from  $\overline{c}_t^h$ by forward recursion
\begin{equation} \label{forward}
\overline{a}_{t+1}^h=\tau \omega_t l^h+(\tau(1+r_t)-\delta)\overline{a}_t^h-\tau \overline{c}_t^h, \quad t=0,...,T, ~~ h=1,...,H.
\end{equation}
Furthermore, the complementary slackness condition has to hold
\begin{equation} \label{comp}
\sum_{h=1}^H\nu_t\overline{a}_t^h=0, ~~t=1,...,T~~ \mbox{and}~~  \sum_{h=1}^H \overline{a}_{T+1}^h=0.
\end{equation}
\end{theorem}
The statements follow from Theorem \ref{thm-noc} by solving (I) for $\overline{c}^h_t$  and in this way replacing the $\lambda_t$ in the forward equation for $\overline{a}_t^h$.

The concept for solving equations (\ref{backward}) - (\ref{comp}) can be used to set up a nonlinear system of equations. The unknown vector is the vector of multipliers $\nu$. The system of nonlinear equations consists of the complementary slackness conditions, where the vectors $a$ depend on $\nu$. This is formulated in corollary \ref{cor-alg}:
\begin{corollary}  \label{cor-alg}
Fix some positive multipliers $\theta \in \R^H$. 
The necessary and sufficient optimality conditions amount to find a nonnegative solution $\nu \in \R^{T+1}$ of a nonlinear system of equations $F(\nu)=0$, where $F: \R^{T+1} \rightarrow \R^{T+1}$ is defined as follows.

Given a vector $\nu \in \R^{T+1}$.
\begin{enumerate}
\item For $T+1$ given unknowns $\nu_1,....,\nu_{T+1}$, compute $\lambda \in \R^{T+1}$ backwards from 
\[  \lambda_T = \nu_{T+1}, \quad  \lambda_{t-1} = \lambda_t(\tau(1+r_t)-\delta) + \nu_t, ~~ t=T,T-1,...,1,
\]
\item Compute forward over time $t=0,...,T$
\[ {a}_{t+1}^h=\tau \omega_t l^h+(\tau(1+r_t)-\delta){a}_t^h-\tau ((u^h)')^{-1} \left( \frac{\tau~ \lambda_{t}}{(\beta^h)^t ~\theta_h} \right), ~~ h=1,...,H.
\]
\item Then the evaluation of $F(\nu)$ is given by
\[ F_t(\nu) = \sum_{h=1}^H\nu_t {a}_t^h, ~~t=1,...,T~~ \mbox{and}~~  F_{T+1} =\sum_{h=1}^H {a}_{T+1}^h.
\]
\end{enumerate}
\end{corollary}
By changing the weights $\theta_h$ we obtain the efficient frontier of the minimal solutions.

%
%
%
%

\section{Generalized Ramsey Model without Default}
Up to this point we made the assumption that at the final time the households have a nonnegative aggregate capital stock 
\[ \sum_{h=1}^H a^h_{T+1} \geq 0, \]
which implies that individual households could have negative capital stock at the end of the time period. 
If one wants to avoid this situation and requires that eventually each individual household has no negative capital stock, we have to change the single final inequality constraint to a set of inequality constraints at time $T$
\[  a^h_{T+1} \geq 0 \quad h=1,..,H. \] 
If one takes a closer look at the equation for the last time period
\[ a_{T+1}^h =\tau \omega_T l^h+(\tau(1+r_T)-\delta)a_T^h-\tau c_T^h  \]
one sees quickly, that $a_{T+1}^h$ cannot be positive in the optimum, since in such a case one could reduce $a_{T+1}^h$ to zero and at the same time increase $c_T^h$ by the same amount which would yield a higher value in the objective function. This contradicts the Pareto optimality. Hence we will use here
\[  a^h_{T+1} = 0 \quad h=1,..,H. \] 

The resulting optimization problem changes slightly from (\ref{Ph}) to
\begin{equation}\label{Phmod}
\begin{split}
\max_{(c^h,a^h)} \sum_{t=0}^T (\beta^h&)^t u^h(c_t^h), \quad h=1,...,H,\\
 a_{t+1}^h&=\tau \omega_t l^h+(\tau(1+r_t)-\delta)a_t^h-\tau c_t^h, ~~ c^h_t > 0,\\
 & \qquad t=0,...,T-1,~ h=1,...,H\\
\quad  0&=\tau \omega_T l^h+(\tau(1+r_T)-\delta)a_T^h-\tau c_T^h, ~~ c^h_T > 0 \quad t=0,...,T,\\
\sum_{h=1}^H a_t^h&\ge 0, ~~ t=1,...,T \\
 \text{with} ~~ \sum_{h=1}^H a_0^h&>0,\ (a_0^h)_{h=1}^H\text{ given}.
\end{split}
\end{equation}
The arguments on existence of solutions and necessary and sufficient optimality conditions can be modified slightly and we can prove existence of minimal points in the same way as for the previous problem. 
However, the optimality conditions change due to the new set of inequality constraints. To interpret these properly we define the modified Lagrangian for this case.
\begin{equation}\label{Lagmod}
\begin{split}
\mathcal{L}(z;\theta,\nu,\varrho)=&-\sum_{h=1}^H  \sum_{t=0}^T \theta_h (\beta^h)^t u^h(c^h_t) - \sum_{h=1}^H\sum_{t=1}^{T} \nu_t a_t^h  \\
& + \sum_{h=1}^H\sum_{t=0}^{T-1} \lambda_t^h (a_{t+1}^h-\tau \omega_tl^h -(\tau(1+r_t)-\delta)a_t^h+\tau c_t^h) \\
& + \sum_{h=1}^H \lambda_T^h (-\tau \omega_T l^h -(\tau(1+r_T)-\delta)a_T^h+\tau c_T^h)
\end{split}
\end{equation}

The theorem on  optimality conditions is rewritten as follows
\begin{theorem} \label{thm-mod}
A vector $\overline{z}$ is Pareto-optimal if and only if there exist multipliers $\theta = (\theta_1,...,\theta_H)\in \R^H$ (where at least one $\theta_h\neq 0$), $\lambda=(\lambda_0^1,...,\lambda_{T}^1,...,\lambda^H_0,...,\lambda_T^H)\in\R^{H(T+1)}$, $\nu=(\nu_1,...,\nu_{T+1})\in \R^{T+1}$ such that
\begin{equation*}
\begin{split}
\text{(I)}\quad\frac{\partial \mathcal{L}(\overline{z};\theta,\nu,\varrho)}{\partial {c}_t^h}&=-\theta_h(\beta^h)^t(u^h)'(\overline{c}_t^h)+\tau \lambda_t^h=0,\quad t=0,...,T,~ h=1,...,H\\[3mm]
\text{(II)}\quad\frac{\partial \mathcal{L}(\overline{z};\theta,\nu,\varrho)}{\partial {a}_t^h}&=-\nu_t+\lambda_{t-1}^h-\lambda_t^h(\tau(1+r_t)-\delta) = 0, \quad t=1,...,T,~h=1,...,H\\[3mm]
\text{and the compl}&\text{ementary slackness condition holds}\\[3mm]
\text{(IV)}\hspace{1.0cm} \nu_t \sum_{h=1}^H \overline{a}_t^h&=0\quad t=1,...,T, \\[3mm]
\text{(V)}\hspace{2.4cm}&\theta_h,\nu_t \ge 0, \quad h=1,...,H, ~~ t=1,...,T.
\end{split}
\end{equation*}
\end{theorem}

In a similar way as for the previous problem, we can collect the information in a different form which also leads to a  system of nonlinear equations.

\begin{theorem} \label{thm-prop1}
Consider the modified optimization problem (\ref{Phmod}) and let $\overline{c}^h_t, \overline{a}^h_t$ be minimal solutions. 
For given initial values $\overline{a}_0^1,....,\overline{a}_0^H$ with $\sum_{h=1}^H \overline{a}_0^h>0$  there exist multipliers $\theta = (\theta_1,...,\theta_H)$, all nonnegative with at least one being positive,
and $\nu=(\nu_1,...,\nu_{T})$, all nonnegative, such that the optimal capital stock $\overline{a}_t^h$ can be computed from the optimal consumption $\overline{c}_t^h$ by forward recursion
\begin{equation} \label{forward1}
\overline{a}_{t+1}^h=\tau \omega_t l^h+(\tau(1+r_t)-\delta)\overline{a}_t^h-\tau \overline{c}_t^h, \quad t=0,...,T-1.
\end{equation}
The optimal consumption $\overline{c}_t^h$ can be obtained by solving recursively backwards for each $h=1,...,H$
\begin{equation} \label{backward1}
\overline{c}^h_{t-1} = ((u^h)')^{-1} \left( \beta^h (\tau(1+r_t)-\delta)(u^h)'(\overline{c}^h_t) +
\frac{\tau~ \nu_t}{ (\beta^h)^{t-1} \theta_h} \right)    \quad t=T, T-1, ..., 1.
\end{equation}
with an unknown final condition.
Furthermore, the complementary slackness condition has to hold
\begin{equation} \label{comp1}
\nu_t \sum_{h=1}^H \overline{a}_t^h=0, ~~t=1,...,T
\end{equation}
and the capital at the final time vanishes
\begin{equation} \label{final}
\tau \omega_T l^h+(\tau(1+r_T)-\delta)\overline{a}_T^h-\tau \overline{c}_T^h = 0, ~~h=1,...H.
\end{equation}
\end{theorem}
The statements follow from Theorem \ref{thm-mod} by solving (I) for $\overline{c}^h_t$  and replacing this way the $\lambda$ in (II).

We can use the  optimality conditions again to set up a system of nonlinear equations. Here, the unknown vector is $(\nu, c_T)$ and the nonlinear equations consist of the complementary slackness conditions and the vanishing final capital stock.  
\begin{corollary}  \label{cor-alg1}
Fix some positive multipliers $\theta \in \R^H$. 
The optimality conditions amount to finding a nonnegative solution $\nu \in \R^{T}$ and a vector $c_T \in \R^H$ of a  system of nonlinear equations $F(\nu,c_T)=0$, where $F: \R^{T+H} \rightarrow \R^{T+H}$ is defined as follows.

Given vectors $(\nu,c_T) \in \R^{T+H}$.
\begin{enumerate}
\item For $T+H$ given unknowns $\nu_1,....,\nu_{T},c_T^1,...,c_T^H$, compute $c_t^h$ backwards from 
\[
{c}^h_{t-1} = ((u^h)')^{-1} \left( \beta^h (\tau(1+r_t)-\delta)(u^h)'({c}^h_t) +
\frac{\tau~ \nu_t}{ (\beta^h)^{t-1} \theta_h} \right)   
\]
for $t=T,T-1,...,1$ and $h=1,...,H$.
\item Compute forward over time $t=0,...,T-1$
\[ {a}_{t+1}^h=\tau \omega_t l^h+(\tau(1+r_t)-\delta){a}_t^h-\tau c_t^h, ~~ h=1,...,H.
\]
\item Then the evaluation of $F(\nu,c_T)$ is given by
\[ F_t(\nu,c_T) = \nu_t \sum_{h=1}^H{a}_t^h, ~~t=1,...,T
\]
and
\[ F_{T+h} (\nu,c_T) = \tau \omega_T l^h +(\tau(1+r_T)-\delta){a}_T^h-\tau {c}_T^h, ~~h=1,...,H
\]
\end{enumerate}
\end{corollary}

\section{Numerical Results}

In order to analyze the economic implications of default or non-default in the generalized Ramsey model with new budget constraint, we implemented both problem formulations and compared the Pareto optimal solutions of the two models.\\

Within this paper, all households have identical preferences according to instantaneous utility. They differ in their initial 
capital endowments and time preferences. We expect that the most patient household, which is the one with the highest $\beta$, will 
consume the most in both models. This is only logical since the time preference is a multiplicative constant in the objective function. 

We consider a coalition of four different households, $100$ time steps and constant interest and wage rates. We assume logarithmic utility preferences for all households, constant labor supply and a depreciation rate greater than zero. In order to solve the optimality conditions, we have to fix the `scalarization'-parameters $\theta^h$ as well. Every Pareto solution will depend on one special choice of these $\theta^h$. Varying these parameters yields the whole Pareto front, but in this framework we will fix them such that all households are weighted equally. 

The parameter constellation is then given by 
\[H=4,\ \theta^h=1/H,\ T=100,\ r=0.03,\ \omega=1,\ l=1,\ \tau=1,\ \delta=0.01\]
and  $u^h(c)=log(c)$ for all $h=1,...,H$. The households differ according to their time preferences, $\beta^h$, 
which we set to 
\[\beta^1=0.9,\ \beta^2=0.93,\ \beta^3=0.95,\ \beta^4=0.98\]
and their initial capital endowments which we fix at
\[a_0^1=30,\ a_0^2=20,\ a_0^3=10,\ a_0^4=10.\]

We obtain the following results for the problem with default:
\begin{figure}[H]
\begin{minipage}{0.5\textwidth}
\includegraphics[width=\textwidth]{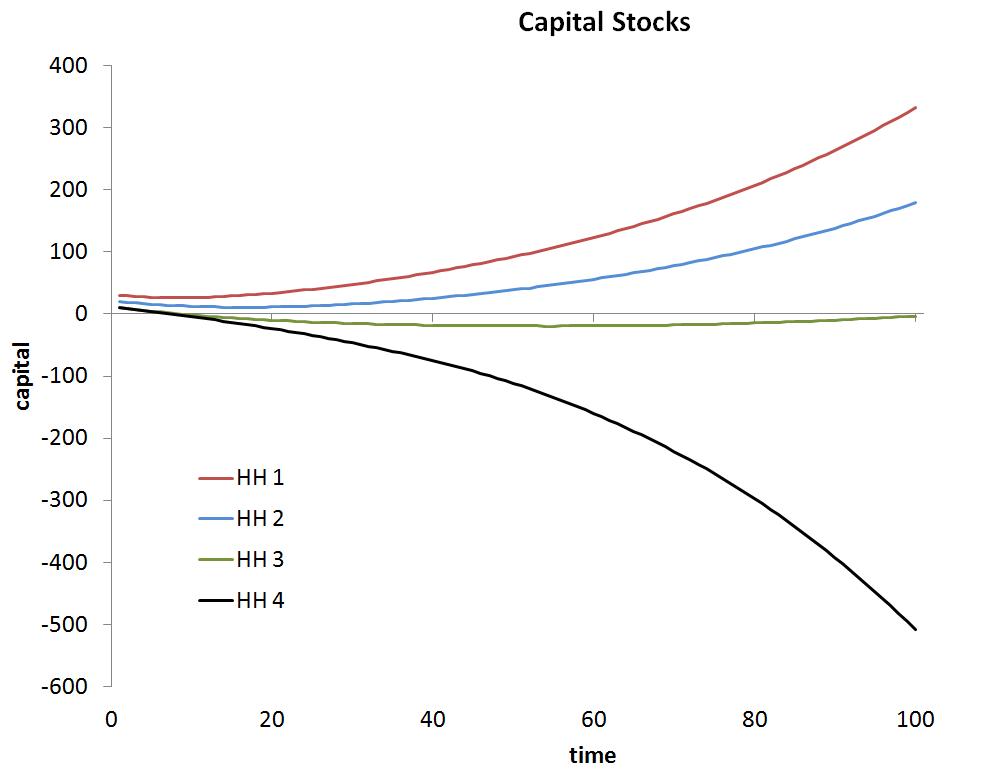}
\caption{Capital Stocks in the Model with Default}
\end{minipage}
\begin{minipage}{0.5\textwidth}
\includegraphics[width=\textwidth]{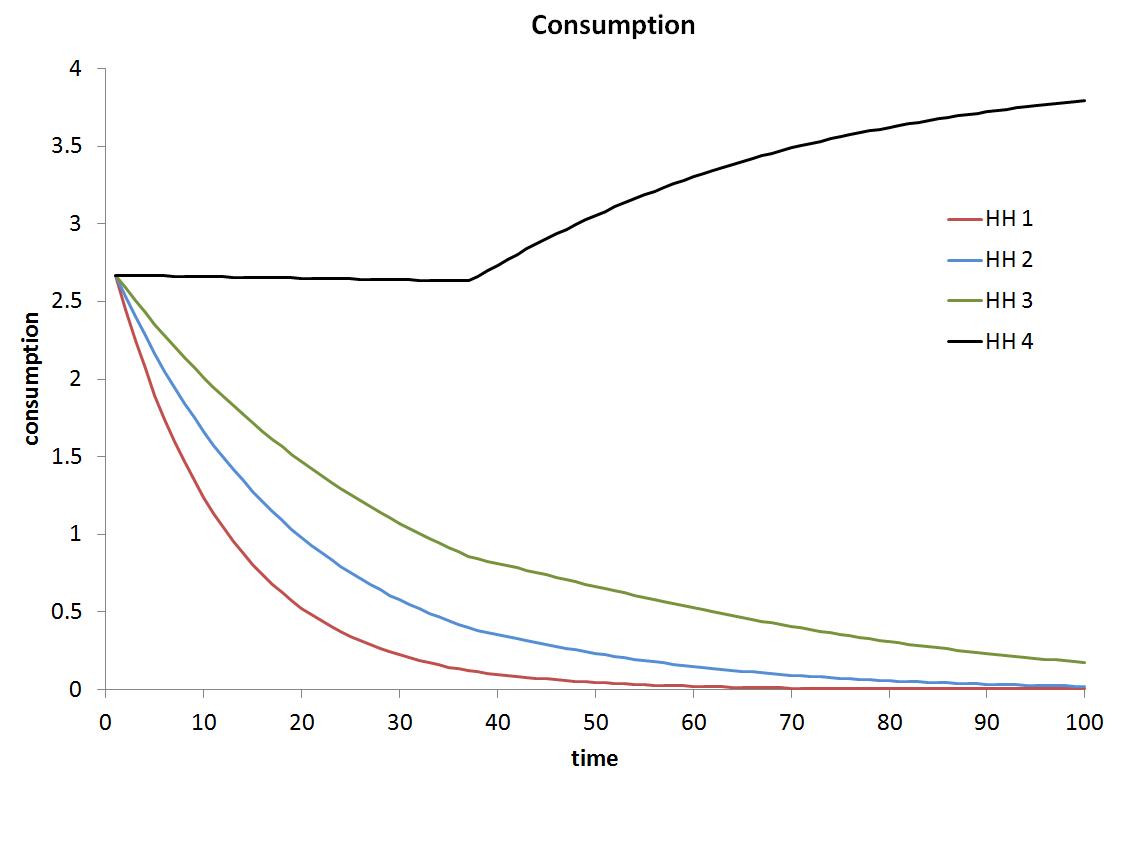}
\caption{Consumption Paths in the Model with Default}
\end{minipage}
\end{figure}
The plot shows that the most patient household (i.e. number $4$) leaves the economy at the end with high debt. He seems to run a Ponzi scheme. Household $3$ only slightly engages in debt. The other two households hold positive capital stocks to fund that debt of $3$ and $4$. We also oberserve (see Figure \ref{aggCap}) that the coalition reduces its initial net-capital stock  at a high rate. After 38 periods, it does not hold positive net-capital stocks anymore. From a macroeconomic perspective, this strategy cannot serve as a role model for the economy. 

Consumption of households $2-4$ turns out as expected, but household $1$ exhibts a strange pattern: As long as the the coalitions net-capital stock is strictly positive, he slightly reduces consumption. Afterwards (i.e. $t \geq 38$) he starts to consume more at a decreasing rate. The households with smaller $\beta$ prefer consumption at the very beginning. The more patient a household is, the more consumption is shifted towards the end (remark that the households are ordered according to the size of their $\beta$ starting with household one). 

In the nondefault model, where every household has to leave the economy at the end with at least a nonnegative capital stock, the behaviour of the households is quite different to the previous case. 
\begin{figure}[H]
\begin{minipage}{0.5\textwidth}
\includegraphics[width=\textwidth]{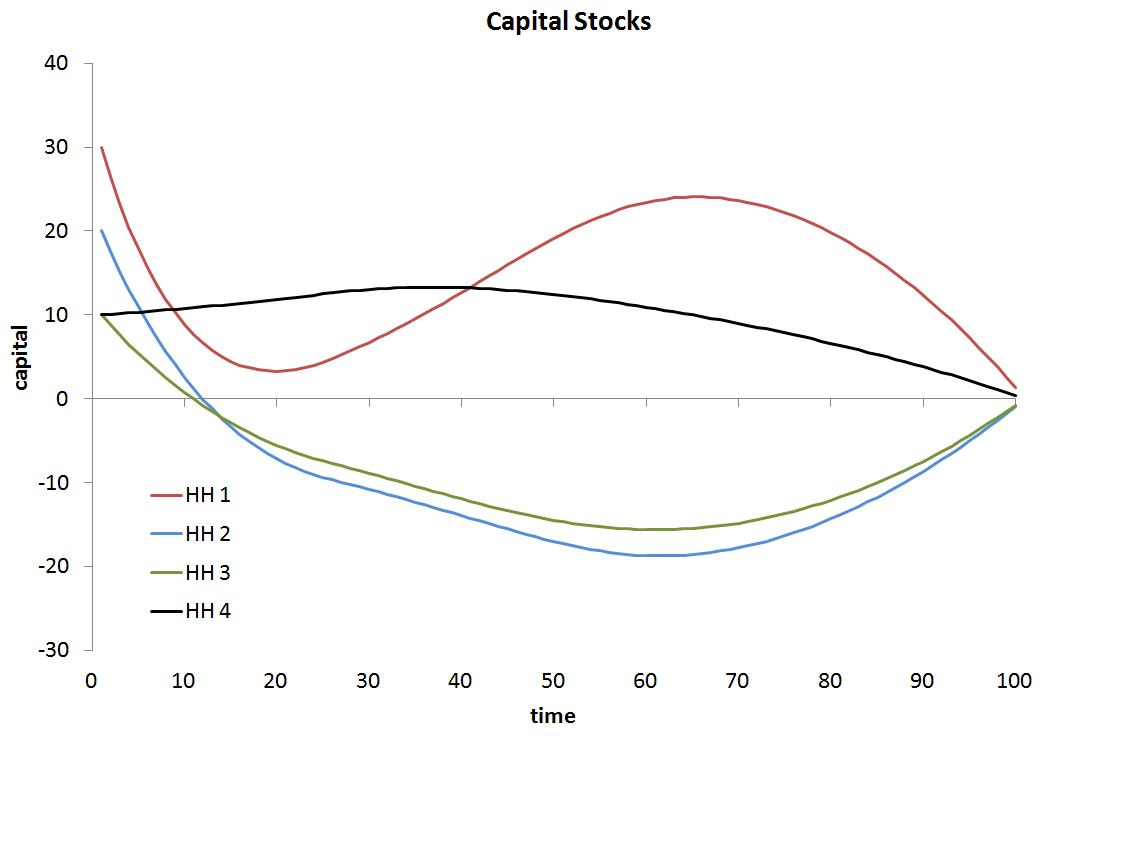}
\caption{Capital Stocks in the Model without Default}
\end{minipage}
\begin{minipage}{0.5\textwidth}
\includegraphics[width=\textwidth]{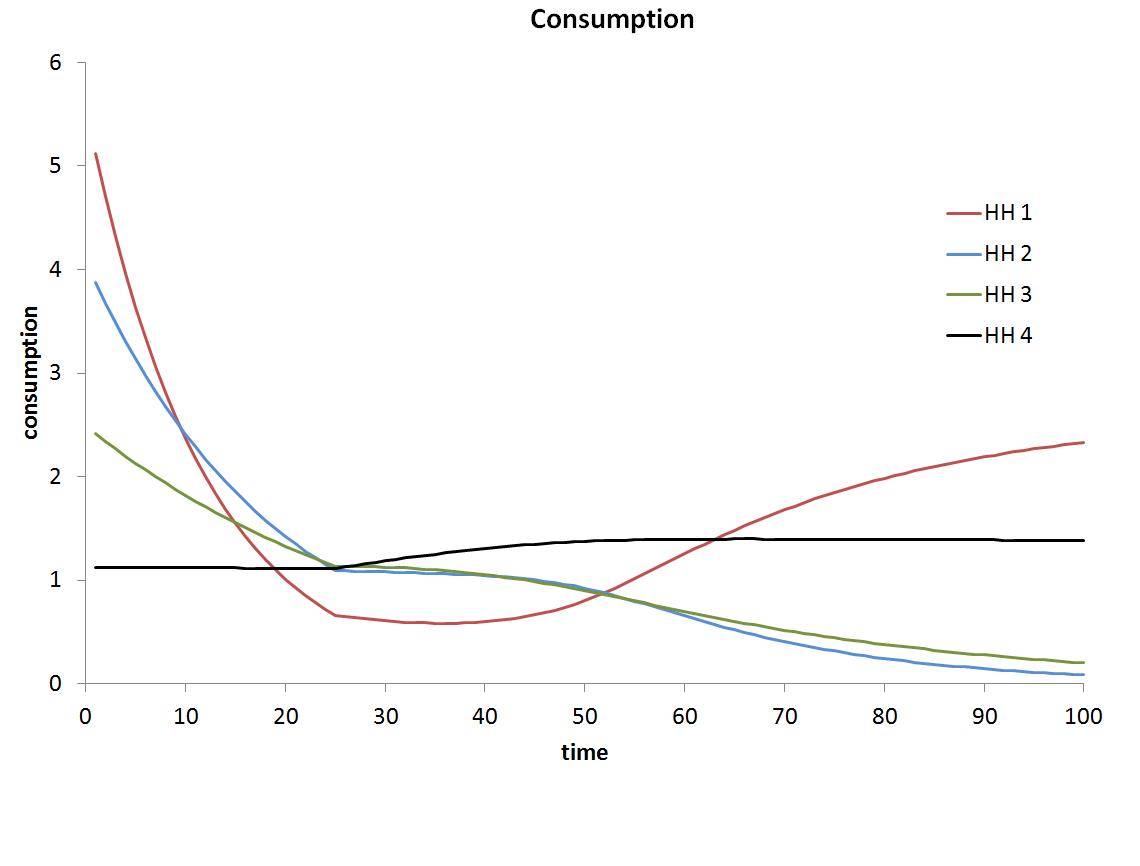}
\caption{Consumption Paths in the Model without Default}
\end{minipage}
\end{figure}

Indeed, there are still some households that finance consumption of others, i.e. they incur debt and do not raise their consumption level, but the households have changed their roles. In the model without default household four, which is the most patient one, and household one hold positive capital stocks in every period. Household two and three now incur debt. 

Different than we expected, the household with the highest $\beta$ has a nearly constant level of consumption. We would expect that its consumption would raise towards the end  since its gain of utility is discounted weaker compared to the others. 
The saving and consumption behavior of household one surprises as well. This household has the smallest time discount factor $\beta$. Hence we would expect it to prefer consumption at the beginning and then reduce consumption towards the end. The initial consumption is indeed very high and the consumption level then decreases quite fast but it rises again after period 50. Another interesting observation is the small difference between household two and three. The development of their capital stocks seems to be quite similar as does their consumption behavior.


Comparing the solutions of both problems, we conclude that permitting agents to default on their debt as long as the coalition is on its budget contraint at any point in time makes a huge difference.  In the model with default the most patient household seems to dominate the less patient ones. It incurs high debt and consumes the most. Moreover, it nearly forces the other households in poverty. This exclusive right to consumption owned by the fourth household vanishes in the model without default. Here, not only time preferences but also initial capital endowments seem to have an impact on the optimal capital and consumption allocation.

If we consider the aggregated capital stocks and consumption paths, both problems are quite similar. Aggregated consumption stays nearly constant over time. The aggregated capital stock declines in both cases. 

\begin{figure}[H]
\begin{minipage}{0.5\textwidth}
\includegraphics[width=\textwidth]{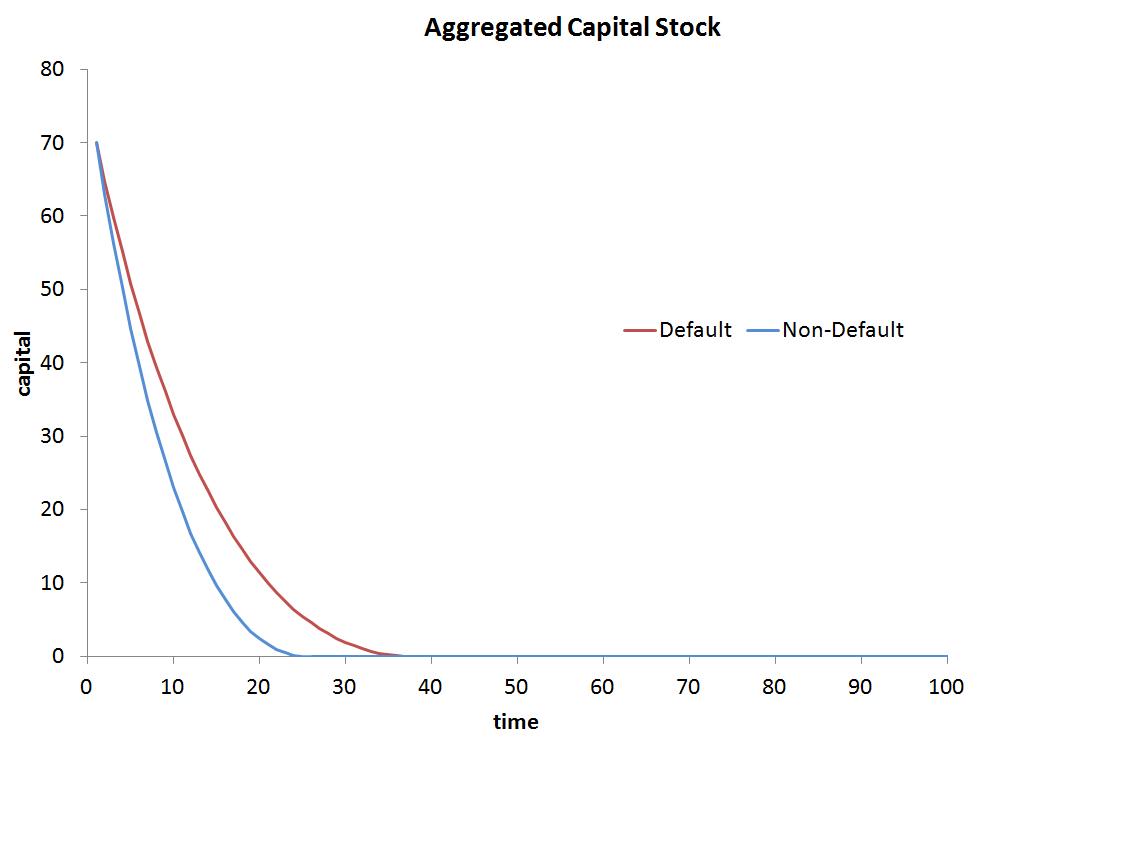}
\caption{Aggregated Capital Stock with and without default}
\label{aggCap}
\end{minipage}
\begin{minipage}{0.5\textwidth}
\includegraphics[width=\textwidth]{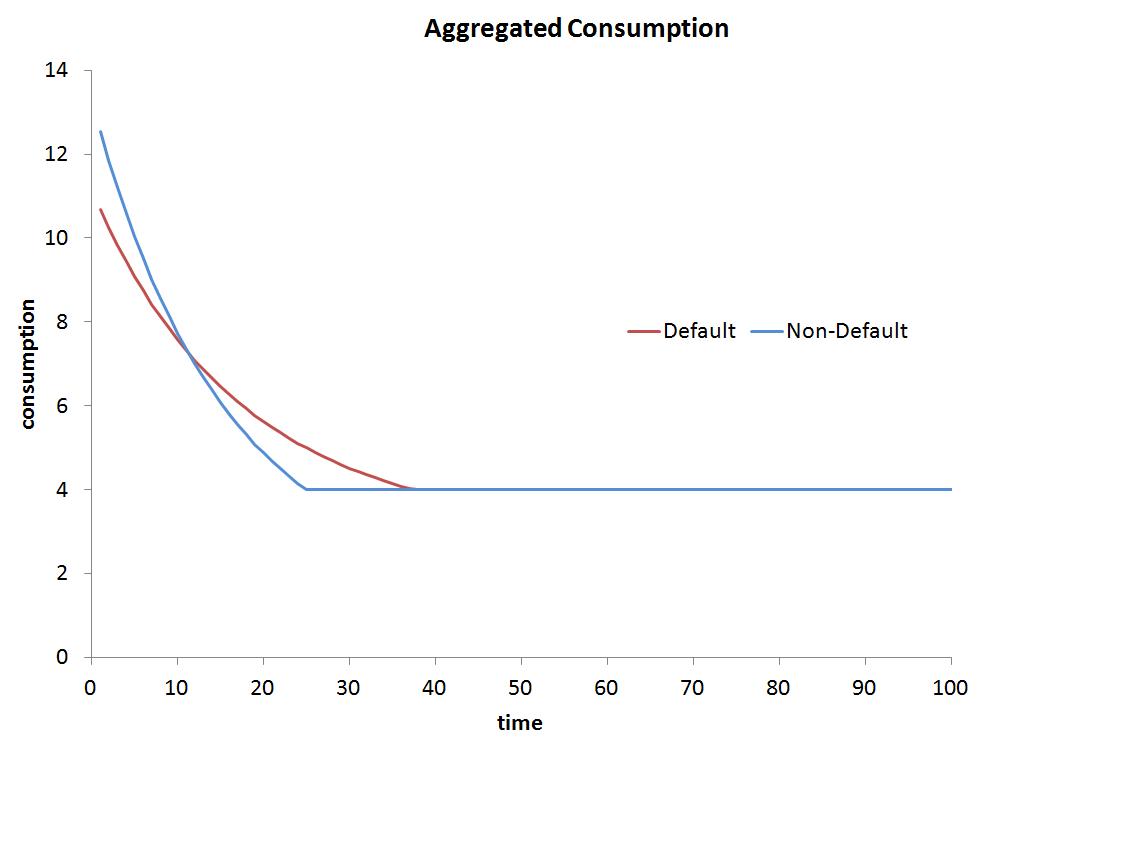}
\caption{Aggregated Consumption with and without default}
\end{minipage}
\end{figure}

\section{Conclusion}

In this paper, we introduce a Ramsey model with an agreed default within a coalition of heterogeneous households. This problem leads to a vector-valued objective function which is treated in the context of vector optimization. We work out a rigorous mathematical framework in order to prove the existence of the utility maximizing allocations of capital and consumption via time for all agents. Necessary and sufficient optimality conditions are derived for this constrained nonlinear multi-objective optimization problem. We adapt this model to the case of no default. As shown in the numerical example, this adaptation can have a huge impact on the behavior of the single households, whereas the Pareto optimal solution of both problems does not change much if the coalition is considered from outside.


 \newpage
\bibliography{literature} 
 \bibliographystyle{apalike}    

\end{document}